\documentclass[12pt]{amsart}
\usepackage{amsfonts,amssymb,amsthm,amsmath,amsxtra,amscd,verbatim,eucal}
\usepackage[all]{xy}
\usepackage[dvips]{graphics}
\usepackage{graphicx}
\usepackage{hyperref}

\theoremstyle{plain}
\newtheorem{thm}{Theorem}[section]

\newtheorem{letthm}{Theorem}

\newtheorem{letcor}[letthm]{Corollary}

\newtheorem{lem}[thm]{Lemma}
\newtheorem{prop}[thm]{Proposition}

\newtheorem{cor}[thm]{Corollary}

\theoremstyle{definition}
\newtheorem{defi}[thm]{Definition}

\newtheorem{question}[thm]{Question}

\theoremstyle{remark}
\newtheorem{eg}[thm]{Example}

\newtheorem{rmk}[thm]{Remark}

\newcommand{\bb}[1]{{\mathbb{#1}}}
\newcommand{\mc}[1]{{\mathcal{#1}}}

\newcommand{\wt}[1]{{\widetilde{#1}}}

\def\D{\bb{D}}

\def\Z{\bb{Z}}
\def\N{\bb{N}}

\def\C{\bb{C}}

\def\R{\bb{R}}

\def\P{\bb{P}}

\def\H{\mc{H}}
\def\cO{\mc{O}}

\def\frm{\mathfrak{m}}

\def\eps{\varepsilon}
\def\om{\omega}

\DeclareMathOperator{\mult} {mult}
\DeclareMathOperator{\PSL} {PSL}
\DeclareMathOperator{\PGL} {PGL}
\DeclareMathOperator{\Aff} {Aff}
\DeclareMathOperator{\ord} {ord}

\DeclareMathOperator{\aut} {Aut}
\DeclareMathOperator{\GL}{GL}

\DeclareMathOperator{\id}{id}

\DeclareMathOperator{\dist}{dist}
\DeclareMathOperator{\kod}{kod}

\newcommand{\abs}[1]{{\left|{#1}\right|}}

\newcommand{\on}[1]{{\operatorname{#1}}}

\title{Normal surface singularities admitting contracting automorphisms}

\date{\today}

\author[C. Favre]{Charles Favre}
\address{CNRS - Centre de Math\'ematiques Laurent Schwartz, 
\'Ecole Polytechnique, 
91128 Palaiseau Cedex, France}
\email{favre@math.polytechnique.fr}

\author[M. Ruggiero]{Matteo Ruggiero}
\address{Fondation Math\'ematique Jacques Hadamard - 
Centre de Math\'ematiques Laurent Schwartz, 
\'Ecole Polytechnique, 
91128 Palaiseau Cedex, France.}
\email{ruggiero@math.polytechnique.fr}

\thanks{First author is supported by the ERC-starting grant project "Nonarcomp" no.307856}

\begin{document}

\begin{abstract}
We show that a complex normal surface singularity admitting a contracting automorphism is necessarily quasihomogeneous.
We also describe the geometry of a compact complex surface arising as the orbit space of such a contracting automorphism.
\end{abstract}

\maketitle
\tableofcontents

\pagebreak


\section{Introduction}

Recent interest arose in describing complex analytic spaces carrying interesting holomorphic dynamical systems.
Automorphisms of compact complex surfaces with non-trivial dynamics have been classified by Gizatullin~\cite{gizatullin:rationalGsurfaces} and Cantat~\cite{cantat:dynamiqueautosurfproj}.
Endomorphisms of compact surfaces have been analyzed in detail in a  recent work of Nakayama~\cite{nakayama:complexnormalprojsurfnonisosurjendo}.
In higher dimensions the situation is less clear, but many progress have been realized in the case of projective varieties \cite{nakayama-zhang:buildingblocksetaleendocplxprojmfld,nakayama-zhang:polarizedendocplxnormalvar,zhang:algvarautogroupsmaxrank}, see also~\cite{horing-peternell:nonalgcpctkahler3fldsendo}.

\smallskip

Here we  focus our attention to a local version of this problem, and more specifically we consider a self-map $f : (Y,0) \to (Y,0)$ of a complex normal \emph{surface} singularity.
Any such space admits non finite self-maps with non-trivial dynamics 
so it is necessary to impose suitable restrictions on $f$ in order to be able to say something on the geometry of the singularity. 
When $f$ is non-invertible and finite, a theorem of Wahl~\cite{wahl:charnumlinksurfsing} states that up to a finite cover
then $(Y,0)$ is either smooth,  simple elliptic or a cusp singularity.
When $f$ is merely assumed to be an automorphism, nothing can be said on $(Y,0)$ 
even if $f$ has infinite order. It is a theorem of M\"uller~\cite[p.230--231]{muller:liegroupsanalyticalgebras} that $(Y,0)$ carries $m$ analytic vector fields  that are in involution and linearly independent for any fixed integer $m\ge1$.
In particular, the automorphism group $\aut(Y)$ is always ``infinite dimensional''. 

We propose here to give a classification of automorphisms of complex normal surface singularities that are \emph{contracting}, in the sense that there exists an open neighborhood $U$ of the singularity $(Y,0)$ whose image by the automorphism is relatively compact in $U$, and $\bigcap_{n=0}^\infty f^n(U)=\{0\}$. 

\smallskip

Before stating our main theorem let us describe some examples of contracting automorphisms.
Pick $(w_1, \ldots , w_n) \in (\N^*)^n$ and suppose we are given a family of weighted homogeneous polynomials $P_1, \ldots , P_k$ satisfying $P_i ( t^{w_1} x_1, \ldots, t^{w_n} x_n) = t^{d_i} P_i(x_1, \ldots, x_n)$ for all $t\in \C^*$ and all $x= (x_1, \ldots , x_n) \in \C^n$ and for some $d_i\in \N^*$.
The map  $f_t (x) = ( t^{w_1} x_1, \ldots, t^{w_n} x_n)$ is then an automorphism that is contracting as soon as  $|t| < 1$ and leaves the analytic space $Y := \bigcap_{i=1}^k \{ P_i = 0 \}$ invariant. 
Such singularities are known as weighted homogeneous (or quasihomogeneous) singularities.
They are characterized as those normal singularities carrying an effective action of $\C^*$ such that all orbits contain the singular point in its closure.
We refer to the survey of Wagreich~\cite{wagreich:structurequasihomogensing} for more detail. 

Our main result can be stated as follows.
\begin{letthm}\label{thm:main}
Suppose $(Y,0)$ is a complex normal surface singularity, and $f: (Y,0) \to (Y,0)$ is a contracting automorphism. 

Then $(Y,0)$ is a weighted homogeneous singularity and 
when $(Y,0)$ is not a  cyclic quotient singularity we have $f^N = f_t$ for some integer $N\ge1$ and some $|t|<1$.
\end{letthm}
We refer to Theorem~\ref{thm:class-admissibledata} for a more precise statement.

Up to an iterate, any contracting automorphism on a  cyclic quotient singularity arises as follows, see the Appendix~A. 
One can write $(Y,0)$ as the quotient of 
$\C^2$ by the automorphism  $\gamma (x,y) = (\zeta x, \zeta^q y)$ for some $m$-th root of unity $\zeta$ with $q$ coprime with $m$; 
and $f$ can be lifted to $\C^2$ under the form
$\wt{f}(z,w)= (\alpha z , \beta w + \eps z^k)$ with $\eps \in \{0,1\}$, $0< |\alpha|, |\beta| < 1$, $\eps (\alpha^k - \beta) =0$, and $\wt{f} \circ \gamma = \gamma \circ \wt{f}$.
One can check that $f$ does not belong to the flow of a $\C^*$-action when $\beta^n \neq \alpha^m$ for all $n, m \in \N^*$.

\smallskip

If $f$ belongs to a reductive algebraic subgroup $G$ of $\aut(Y)$, then our result would follow 
since the Zariski closure of $f$ in $G$ is a subgroup isomorphic to $\C^*$ acting effectively on $Y$, and a theorem of Scheja and Wiebe~\cite[Satz 3.1]{scheja-wiebe:chevalleyzerlegungderiv} implies $Y$ to be a weighted homogeneous singularity, see~\cite[Theorem~1]{muller:symmetriessurfsing}.
If $f$ is the time-one map of the flow of a holomorphic vector field having a dicritical singularity at $0$,  our theorem is then due to Camacho, Movasati and Scardua~\cite{camacho-movasati-scardua:quasihomosteinsurfsing}.

\smallskip

In our situation, we only have a map at hands and we shall therefore  base our analysis
on the dynamical properties of this map. One can summarize our strategy as follows.

First we prove that the dual graph of the minimal resolution of $(Y,0)$ is star-shaped (Theorem~\ref{thm:starshaped}). 
When this graph is a chain of rational curves, then $(Y,0)$ is a cyclic quotient singularity and it is not difficult to conclude.
 Otherwise, the unique branched point of the dual graph is 
a component $E$ to which are attached finitely many chains of rational curves that we may contract to cyclic quotient singularities. Since $f$ is contracting, one can construct a natural 
foliation by stable disks that are all transversal to $E$. Linearizing $f$ on each stable disk allows one to endow them with a canonical linear structure, so that we may linearize the embedding of $E$ in the ambient surface. In other words, a tubular neighborhood of $E$ is analytically isomorphic to a neighborhood of the zero section in the total space of a suitable line bundle $L \to E$ of negative degree. This implies $(Y,0)$ to be a cone singularity, which implies our theorem. 

Actually one subtlety arises from the presence of quotient singularities on $E$. To deal with this problem, it is convenient to endow $E$ with a natural structure of orbifold.
One can linearize a tubular neighborhood of $E$ exactly as before, except that 
$L$ is in general only an orbibundle\footnote{line orbibundles are also known as line V-bundles (see e.g.~\cite{satake:gaussbonnetVmflds}).}.
The rest of the argument remains unchanged and we conclude that the singularity is weighted homogeneous.

\smallskip

Our dynamical approach gives an alternative proof of the result of Orlik and Wagreich~\cite{orlik-wagreich:isolatedsingagsurfCaction} that the dual graph of the minimal desingularization of a weighted homogeneous singularity is star-shaped.
The original approach of op. cit. was very much topological in nature and relied on the classification of $S^1$-action on Seifert $3$-manifolds, whereas the (short) proof given by M\"uller in ~\cite{muller:resolweighethomosurfsing} is based on Holmann's slice theorem~\cite[Satz 4]{holmann:quotkomplexautogrup} and uses in an essential way the $\C^*$-action on such a singularity. 

\smallskip

Since a contracting automorphism acts properly on a punctured neighborhood of the singularity, 
we may form the orbit space 
$S(f) := (Y\setminus \{ 0 \}  ) / \langle f \rangle$ which is a compact complex surface.
Our second result gives a precise description of the geometry of this surface.
To state it properly we recall some terminology.

A Hopf surface is a compact complex surface whose universal cover is isomorphic to $\C^2 \setminus \{ 0 \}$.
An isotrivial  elliptic fibration is a proper submersion $\pi:  S \to B$ from a compact complex surface to a Riemann surface
whose fibers are all isomorphic to a given elliptic curve $E$. When there is an action of $E$ on $S$ preserving $\pi$ which induces 
a translation on each fiber of $\pi$, then we say that $S$ is a principal elliptic fibre bundle.
A Kodaira surface is a non-K\"ahler compact surface of Kodaira dimension $0$, see e.g.~\cite[\S VI]{barth-hulek-peters-vanderven:compactcomplexsurfaces}.
Any Kodaira surface can be obtained as a principal elliptic fibre bundle over an elliptic base or as a quotient of it by a cyclic group.

\begin{letcor}\label{cor:mainsurfaces}
Let $f : (Y,0) \to (Y,0)$ be a contracting automorphism of a complex normal surface singularity.
\begin{enumerate}
\item 
The surface $S(f)$ is a non-K\"ahler compact complex surface.
\item
Either $S(f)$ is isomorphic to a Hopf surface; or $S(f)$ is the quotient of a principal elliptic fibre bundle $S$
by a finite group acting freely on $S$ and preserving the elliptic fibration.
\item
When $\kod (S(f)) =-\infty$, then $S(f)$ is a Hopf surface; 
when $\kod (S(f)) = 0$, it is a Kodaira surface;  otherwise $\kod (S(f)) = 1$.
\end{enumerate}
\end{letcor}

We observe that Corollary~\ref{cor:mainsurfaces} (a) and (c) can be obtained directly as follows.
Indeed,  any normal surface singularity admits a strongly pseudoconvex neighborhood $U$ by~\cite{grauert:ubermodifikationen}, and
the image of $\partial U$ on $S(f)$ defines a global strongly pseudoconvex shell in $S(f)$ in the sense of \cite{kato:cpctcplxsurfwithGSPH}.  Theorem p.538 of op. cit. thus applies and statements (a) and (c) of our
corollary follow.  Our proof however follows a very different path. 
We do not rely on Kodaira's classification of
compact complex surfaces, and show constructively the existence of an isotrivial elliptic fibration.

\medskip

It is tempting to ask whether the results of this paper can be generalized to higher dimensions. 
In particular we may ask the following questions.
\begin{question}
Is it true that any complex normal isolated singularity $(Y,0)$ admitting a contracting automorphism admits a proper $\C^*$-action?
\end{question}
A special interestingr case of the above question is the following
\begin{question}
Does a cusp singularity in the sense of Tsuchihashi support any contracting automorphism?
\end{question}
We refer to~\cite{oda:convexbodies,tsuchihashi:higherdimcuspsing} for a definition of cusp  singularities, and to~\cite{sankaran:higherdiminouehirzebruchsurfaces} for the description of interesting subgroups of local automorphisms 
of these singularities.
Observe that these singularities carry finite endomorphisms that are not automorphisms, see~\cite[\S 6.3]{boucksom-defernex-favre:volumeisolatedsing}.
\begin{question}
Is the orbit space of any contracting automorphism of a complex normal isolated singularity $(Y,0)$ non-K\"ahler?
\end{question}

\subsubsection*{Acknowledgements}
We warmly thank Serge Cantat for useful discussions.


\section{Admissible data}

Suppose that $f: (Y,0) \to (Y,0)$ is a contracting automorphism of a normal surface singularity. Then $f$ lifts to an automorphism  $F: X \to X$ of the minimal resolution $\pi: X \to Y$ of $(Y,0)$, see~\cite[Theorem 5.12]{laufer:normal2dimsing}.
In the course of the proof of our main Theorem \ref{thm:main} we shall replace several times the model $X$ by other more adequate bimeromorphic models.
In order to deal with these modifications, we shall consider the triple $(X,W,F)$ with $W = \pi^{-1}(0)$, and  it is convenient to introduce the following terminology. 
\begin{defi}
An \emph{admissible data} is a triple $\H=(X,W,F)$, such that
\begin{enumerate}
\item $X$ is a (non-compact) connected normal complex surface;
\item $W$ is a connected compact analytic subset of $X$;
\item $F : X \to X$ is a biholomorphism onto its image;
\item $F(X) \Subset X$, $F(W) = W$ and $\bigcap_{n\ge 0} F^n(X) \subseteq W$;
\item the intersection form of $W$ is negative definite.
\end{enumerate}
\end{defi}
Observe that Condition (d) is saying that $W$ is an attractor whose basin contains $X$ (see~\cite{milnor:conceptattractor} for definitions).

Thanks to Grauert's contraction criterion, any admissible data $(X,W,F)$ gives rise to a contracting automorphism on a normal surface singularity obtained by contracting $W$ to a point.  If the resulting contracting automorphism is conjugated to $f : (Y,0) \to (Y,0)$ then we say that 
the admissible data is \emph{compatible} with $(Y,0,f)$.

\begin{defi}
Two admissible data $(X,W,F)$ and $(X',W',F')$ are said to be \emph{equivalent} if there exists a neighborhood $X_1$ of $W$ and $X_1'$ of $W'$ and a bimeromorphic map $\phi : X_1 \dashrightarrow X_1'$ such that:
\begin{itemize}
\item $\phi$ induces an isomorphism from $X_1\setminus W_1$ onto $X'_1\setminus W'$,
\item $(X_1,W,F)$ and $(X'_1,W',F')$ both define an admissible data, and
\item we have $F' \circ \phi = \phi \circ F$.
\end{itemize}
\end{defi}

Classifying  admissible data modulo equivalence is the same as classifying contracting automorphisms on normal surface singularities modulo conjugacy.


\section{Local Green function}
Suppose $(Y,0)$ is a normal surface singularity, and $f:(Y,0) \to (Y,0)$ is a contracting automorphism.
In this section we quantify the condition of being contracting  by introducing a suitable function that is strictly decreasing under iteration of $f$.
Before stating precisely our main statement we recall a few facts about canonical desingularization of surfaces.


\subsection{Canonical desingularization}\label{ssec:desingularization}

Denote by $\frm = \frm_{Y,0}$ the maximal ideal of $\cO_{Y,0}$.
Recall a log-resolution of $\frm$ is a bimeromorphic map $\pi : X \to Y$ from a smooth surface $X$ to $Y$ that is an isomorphism above $Y\setminus\{0\}$ and such that the pull-back of the ideal $\frm$ to $X$ is locally principal at each point on the exceptional divisor.
We also impose that the exceptional divisor $\pi^{-1}(0)$ has \emph{simple normal crossings}, i.e., any irreducible component of $\pi^{-1}(0)$ is smooth, any two such components intersect transversely, and no three distinct components have a common intersection.

\begin{prop}\label{prop:regularization}
Let $(Y,0)$ be any normal isolated surface singularity.
There exists a log-resolution $\pi : X \to Y$ of $\frm$ such that any automorphism of $Y$ lifts to an automorphism of $X$.
\end{prop}

For convenience, we shall say that $X$ is a \emph{nice resolution} of $Y$ if it satisfies the condition of the proposition.

\begin{rmk}
It is not the case that the minimal resolution of a normal surface singularity is always a log-resolution of its maximal ideal. 
Pick any ample line bundle $L\to E$ over a compact Riemann surface. 
The total space of its dual  $L^{-1}$ is the minimal resolution of the cone singularity $(Y,0)$ obtained by contracting its zero section.
Observe that $Y$ is isomorphic to the normalization of $\textup{Spec} \left( \oplus_{n\ge 0} H^0(E, L^{\otimes n})\right)$.
When $h^0(E,L) \neq 0$ and $L$ has some base points, then $\frm_{Y,0}$ is not locally principal in $|L^{-1}|$.
This happens for instance when the degree of $L$ is  equal to the genus of $E$.
\end{rmk}

\begin{proof}
Since the maximal ideal $\frm$ is preserved by any automorphism, it follows that $\aut(Y)$ lifts to the normalized blow-up $X'$ 
of the maximal ideal.
It is possible that $X'$ has some singularities.
We may then consider the minimal (good) desingularization of each of these points.
In this way we get a smooth surface $X''$ such that $\frm\cdot \cO_{X''}$ is locally principal and $\aut(Y)$ lifts to $X''$.
Finally the exceptional divisor might not have simple normal crossing singularities. 
But the singularities of this divisor are fixed by $\aut(X'')$ hence we may resolve them by keeping the property of lifting $\aut(Y)$.
This concludes the proof.
\end{proof}


\subsection{Local Green functions}

Put any euclidean distance on $X$.
In this section, we prove the following result.

\begin{prop}\label{prop:greenfnctcontracts}
Let $(X,W,F)$ be any admissible data.
Suppose $X$ is smooth and $W$ has simple normal crossings.
Then there exists a function $g : X \to [-\infty, 0]$ 
that is smooth on $X\setminus W$, and satisfies the following conditions:
\begin{itemize}
\item $g$ is bounded on $X\setminus F^n(X)$ for all $n\ge 1$;
\item $g (p) \le \log \dist (p,W)$ for all $p\in X$;
\item for any constant $C>0$, there exists an integer $N\ge1$ such that $g \circ F^N \leq g - C$ on $X$.
\end{itemize}
\end{prop}
\begin{rmk}\label{rmk:reduction}
Observe that if the statement is true for an admissible data $(X,W,F)$, then it is true for any other equivalent admissible data $(X',W',F')$.
Indeed, if the equivalence is given by a map $\phi : X' \dashrightarrow X$, then we may set $g' := g \circ \phi$ in a neighborhood of $W'$ and extend it to $X'$ by taking $\min \{ g', - A \}$ for some sufficiently large positive constant $A$. 
\end{rmk}
Observe that this result implies
\begin{cor}
For any admissible data $(X,W,F)$ there exists a basis of neighborhoods $X_n$ of $W \subset X$ 
such that $F(X_n) \subset X_n$ hence $(X_n,W,F|_{X_n})$ remains an admissible data.
\end{cor}
\begin{proof}
Observe that $\exp(g)$ is continuous on $X$.
Define $D := \sum_{n\ge 0} \exp (g \circ F^n)$. The previous proposition applied to $C = 1$ yields
$$
\exp (g \circ F^n) \leq \exp \left( -\left\lfloor \frac{n}{N} \right\rfloor \right) \, \sup_{0\leq l \leq N-1} \exp ( g \circ F^l)
\leq \exp \left( -\left\lfloor \frac{n}{N} \right\rfloor\right)
$$
for any $n\ge 0$.
This implies the series defining $D$ converges uniformly to a continuous function vanishing at $0$. Now we have $D \circ F = D - \exp (g \circ F) < D$, so that
we can take $X_n := \{ D < \frac{1}{n}\}$.
\end{proof}

The rest of this section is devoted to the proof of Proposition \ref{prop:greenfnctcontracts}.
We first need some preliminary results.

\begin{lem}\label{lem:constrpsh}
Let $(Y,0)$ a normal surface singularity, and $f:(Y,0) \to (Y,0)$ a contracting automorphism.
Then there exists a compatible admissible data $(X,W,F)$ satisfying the following conditions. 
\begin{itemize}
\item 
$X$ is a nice resolution of $(Y,0)$;
\item
there exists a psh function $g_0 : X \to [-\infty, 0]$ such that  $dd^c g_0 = \om + [Z]$ where $\om$ is a smooth positive form and $[Z]$ is a current of integration whose support is equal to $W$.
\end{itemize}
\end{lem}
\begin{proof}
Let $\frm:=\frm_{Y,0}$ be the maximal ideal at $0 \in Y$.
Let $\pi:X \to (Y,0)$ be any nice resolution of $(Y,0)$, whose existence is given by Proposition \ref{prop:regularization}.

Pick any finite set of holomorphic maps $\chi_1, ... , \chi_r$ generating $\frm_{Y,0}$ and defined in some common open neighborhood $U$ of $0$.
Restricting $U$ if necessary we may assume that $\sup_i |\chi_i| \le 1$ for all $i$, 
$f(U) \Subset U$, and $\cap_{n\ge 0} f^n(U) = \{0\}$.
We now set
\begin{equation}\label{eqn:defgfrm}
g_Y (y) := \log \left( \sum_{i=1}^r |\chi_i(y)|^2 \right) - \log r~.
\end{equation}
This is a $[-\infty, 0]$-valued function that is smooth and plurisubharmonic on $U \setminus \{0\}$ and tends to $-\infty$ when the point tends to $0$.
Since $X$ is a log-resolution of $\frm$, then locally at each point $p\in \pi^{-1}(0)$ we may find local coordinates $(z,w)$ such that $\frm \cdot \cO_{X,p}$ is locally generated by $z^a w^b$ for some $a,b\ge0$.
It follows that the function
\begin{equation}\label{eq:estimgreen}
g_Y \circ \pi (z,w) - a \log|z| - b \log|w|
\end{equation}
is a smooth function.
We conclude by setting $g_0= g_Y \circ \pi$. 
\end{proof}

By Remark~\ref{rmk:reduction} we may hence suppose the existence of a non-negative psh function $g_0$ on $X$ satisfying the conclusion of the previous lemma.

Let us introduce the following set:
$$
\mc{G} :=\{h \in \on{Psh}(X), \, h \le0 , \text{ and }  \nu(h,p) \geq \nu(g_0,p) \text{ for any } p \in W\}~,
$$
where $\nu(h,p)$ stands for the Lelong number associated to the psh function $h$, see e.g.~\cite{demailly:mongeamperelelongnumbersintertheory}.
Observe that $g_0$ belongs to $\mc{G}$. We now follow the construction of upper envelope in the spirit of~\cite{rashkovskii-sigurdsson:greenfcntwithsing}.

\begin{prop}\label{prop:supispsh}
The function $g_1 := \sup\{ h \in \mc{G}\}$ is a psh function  such that $g_1 - g_0$ is bounded.
\end{prop}
\begin{proof}
Since the maximum of two plurisubharmonic functions $h_1, h_2$ is still psh and 
$\nu(\max \{ h_1, h_2\} , p ) = \max \{ \nu (h_1,p), \nu (h_2,p)\}$ for any point $p$
it is clear that $\mc{G}$ is stable by taking maximum.

By Choquet's lemma, one can find an increasing sequence $h_n \in \on{Psh}(X)$ such that $h_n \to g_1$.
Let $g^\star (p) := \varlimsup_{p' \to p} g_1(p')$ be the upper semi-continuous regularization of $g_1$.
Then $g^\star$ is psh.
Lelong numbers are upper semicontinuous along increasing sequences hence $\nu(g^\star, p)\ge \nu (g_0,p)$ for all $p\in W$.
It follows that $g^\star$ belongs to $\mc{G}$ whence $g_1= g^\star\in \on{Psh}(X)$.

Since $g_1 \ge g_0$, we have $\nu(g_1,p) \le \nu(g_0,p)$ for any $p\in W$, hence $\nu(g_1,p) = \nu(g_0,p)$.
Pick now local coordinates $(z,w)$ at $p \in W$ such that $W=\{z=0\}$ (the case where the exceptional divisor is reducible at $p$ can be treated analogously).
Then $g_0(z,w)= a \log \abs{z} + \cO(1)$, and $g_1(z,w) \le a \log \abs{z} + \cO(1)$, which implies $g_1\le g_0 + \cO(1)$ in a neighborhood of $p$.
By the compactness of $W$, we conclude that $g_1- g_0 \ge 0$ is also bounded from above everywhere in $X$ as required. 
\end{proof}

\begin{proof}[Proof of Proposition~\ref{prop:greenfnctcontracts}]
Since we have $F(X) \Subset X$, and $\bigcap_n F^n(X) \subseteq W$, and $g_0(p)$ tends to $- \infty$ when $p$ tends to $W$, it follows that for any positive constant $C_1>0$ there exists an integer $N \geq 1$ such that $g_0 \circ F^N \leq  -C_1$.

The map $F$ is a biholomorphism preserving $W$, hence for a suitably divisible integer $N$ (for example if $F^N$ fixes all the components of $W$) we have $(F^N)^* [ Z ] = [Z]$.
It follows that $(F^N)^* (dd^c g_0) - [Z] = dd^c ( g_0 \circ F^N) - [Z] $ is a smooth form.
We conclude that $g_0 \circ F^N +  C_1$ belongs to $\mc{G}$, hence $g_0 \circ F^N \le g_1 -  C_1 \le g_0 + \sup |g_1 - g_0| - C_1$ which implies the result for $g:= g_0$.
\end{proof}


\section{Geometry of the dual graph}

Our standing assumptions in this section are the following: $(X,W,F)$ is an admissible data, $X$ is smooth, $W$ has simple normal crossings, $F: X \to X$ is holomorphic and fixes any irreducible component of $W$.


\subsection{Main statement} 
Let us begin by introducing some convenient terminology.

\begin{defi}
An automorphism $h: E \to E$ of a compact Riemann surface $E$ is said to be hyperbolic if $E$ is the Riemann sphere and $h$ has one contracting and one repelling fixed points.
\end{defi}

\begin{defi}
Let $X$ be any smooth complex surface.
A \emph{cycle} (resp., \emph{chain}) of rational curves on $X$ is a finite collection $E_1, ... , E_n$ of smooth rational curves
intersecting transversally whose dual graph is a circle (resp., a segment).
\end{defi}

\begin{defi}
A graph is said to be \emph{star-shaped} if it is a tree (i.e., it is homotopically trivial) and admits at most one branched point. 
\end{defi}
By convention a segment is star-shaped.

This section is devoted to the proof of the following result.
\begin{thm}\label{thm:starshaped}
Suppose $(X,W,F)$ is an admissible data, $X$ is smooth, $W$ has simple normal crossings, and $F$ fixes any irreducible component of $W$.

Then the dual graph $\Gamma(W)$ of $W$ is star-shaped.
More precisely, we are in one and exactly one of the following two situations.
\begin{enumerate}
\item The support of $W$ is a chain of rational curves, and $F|_E$ is hyperbolic for any component $E \subseteq W$.
\item There exists a component $E_\star$ such that $F|_{E_\star}$ is not hyperbolic.
In this case, $F$ is hyperbolic on any other components and the closure of $W \setminus E_\star$ is the disjoint union of a finite number of chains of rational curves.
Moreover if $E_\star$ is not a rational curve, then $F|_{E_\star}$ has finite order.
\end{enumerate}
\end{thm}


\subsection{Proof of Theorem \ref{thm:starshaped}}

We shall proceed in four steps.
The first three steps are devoted to the case when $X$ is a nice resolution of $(Y,0)$.
We explain how to reduce the theorem to this case in the fourth and last step.

\smallskip

\noindent {\bf Step 1}. Assume $\pi:X \to (Y,0)$ is a nice resolution of $(Y,0)$, and suppose that $F|_E$ is hyperbolic for all irreducible components $E \subseteq W$.

Then all components are rational curves.
Since a hyperbolic map has exactly two periodic points, it follows that a component can intersect at most two other components.
In particular the dual graph is either a segment or a circle. 

\begin{prop}\label{prop:nocusps}
Suppose $(X,W,F)$ is an admissible data as above.

Suppose $F|_E$ is hyperbolic for any irreducible component of $W$.
Then $W$ cannot be a cycle of rational curves.
\end{prop}
This proves that $W$ is a chain of rational curves, and we get case (a) of Theorem \ref{thm:starshaped}.
\begin{rmk}
A singularity for which the exceptional divisor of its desingularization is  cycle of rational curves is called a cusp. 
Interesting subgroup of automorphisms of cusp singularities have been constructed by Pinkham in~\cite{pinkham:autocuspsinouehirzebruchsurf} in connection with the automorphism group of compact complex surfaces of Inoue-Hirzebruch's type.
\end{rmk}
We rely on the following lemma.
\begin{lem}\label{lem:eigenvalonintersections}
Pick any point $p \in W$ that is fixed by $F$. 
\begin{itemize}
\item
When $p$ belongs to a unique exceptional component $E$, then 
$dF(p)$ admits a eigenvalue of modulus $<1$ whose eigenvector is transverse to $E$.
\item
When $p$ is the intersection point of two irreducible components $E, E'$ of $W$, then we have
\begin{equation}\label{eq:eigenvalonintersecestimate}
\frac{\log \abs{dF|_E (p)}}{a_E} +\frac{\log \abs{dF|_{E'} (p)}}{a_{E'} } <0,
\end{equation}
where $a_E = \ord_E (\frm \cdot \cO_Y)$ is the order of vanishing of $\pi^*\frm$ along $E$, and $a_{E'} = \ord_{E'} (\frm \cdot \cO_Y)$.
\end{itemize}
\end{lem}
\begin{proof}
We treat only the second case, the first being completely analogous and easier.
We denote by $g$ the Green function given by Proposition \ref{prop:greenfnctcontracts}.

Pick local coordinates $(z,w)$ at $p$ such that $E=\{z=0\}$ and $E'=\{w=0\}$
In these coordinates we have
$$
\big|g(z,w) - a_E \log\abs{z} - a_{E'} \log\abs{w} \big| \le C_1 
$$
for some $C_1 >0$, see \eqref{eq:estimgreen}.
Notice that here we use the fact that $X$ is a nice resolution of $(Y,0)$.

Set $\lambda=(dF|_{E'})(p)$ and $\mu=(dF|_{E})(p)$.
Then for any integer $N\ge0$, one can write
$$
F^N (z,w) = \big(\lambda^N z (1+ \eps_N), \mu^N w ( 1+ \eta_N)\big)
$$
with $\eps_N (0) = \eta_N(0) =0$.
Pick $C>2 C_1$. 
By Proposition \ref{prop:greenfnctcontracts}, for any $C > 0$ there exists $N \gg 0$ so that
$$
g( F^{N} (z,w) )\le g(z,w) -C.
$$
It follows
$$
a_{E}\log\abs{\lambda^N z} +  a_{E'}\log\abs{\mu^N w} + \cO(z, w) - C_1
\le
a_{E}\log\abs{z} +  a_{E'}\log\abs{w} + C_1 - C,
$$
and hence
$$
N \big(a_E \log \abs{\lambda} + a_{E'} \log \abs{\mu}\big) \leq 2C_1 - C + \cO(z,w).
$$
Letting $(z,w) \to 0$, we get $a_E\log\abs{\lambda} + a_{E'}\log\abs{\mu} < 0$.

We conclude by dividing the last relation by $a_E a_{E'}>0$.
\end{proof}

\begin{proof}[Proof of Proposition \ref{prop:nocusps}]
Enumerate the exceptional components $E_0, ... , E_{n-1}$ in such a way that 
$E_i \cdot E_j = 1$ if and only if $\abs{i - j} = 1$, and set $E_n := E_0$. 
Set $p_j=E_j \cap E_{j+1}$ for $j=0, \ldots, n-1$, and write $\lambda_j = dF|_{E_j}(p_j)$ and $\mu_j = dF|_{E_{j+1}}(p_j)$.
Since $F|_{E_j}$ is hyperbolic it follows that $\lambda_j = \mu_{j-1}^{-1}$.

Set $a_j:= \ord_{E_j} (\frm \cdot \cO_Y)$. By Lemma \ref{lem:eigenvalonintersections} we get
$$
0> 
\sum_{j=0}^{n-1}
\left(\frac{\log\abs{\lambda_j}}{a_j}+  \frac{\log\abs{\mu_j}}{a_{j+1}}\right)
=
\sum_{j=0}^{n-1}
\left(\frac{\log\abs{\lambda_j}}{a_j}- \frac{\log\abs{\lambda_{j+1}}}{a_{j+1}}\right)
=0,
$$
a contradiction.
\end{proof}

\begin{rmk}
In the paper of Camacho-Movasati-Scardua~\cite{camacho-movasati-scardua:quasihomosteinsurfsing}, our argument is replaced by a suitable use of the Camacho-Sad index formula for holomorphic vector fields on complex surfaces.
\end{rmk}

\medskip

\noindent {\bf Step 2}. We now prove  that there exists at most one component $E_\star$ for which $F|_{E_\star}$ is not hyperbolic.
More precisely we aim at
\begin{prop}\label{prop:dualgraphstructure}
Suppose $F|_{E_\star}$ is not hyperbolic.
Let $d$ be the graph metric on the vertices of the dual graph $\Gamma(W)$ of $W$.
Pick any component $E \neq E_\star$. Then 
\begin{enumerate}
\item $E$ is rational, and $F|_E$ is hyperbolic;
\item there exists a unique component $E'$ intersecting $E$ s.t. $d(E_\star, E') =  d(E_\star, E) -1$;
\item if $p = E \cap E'$, then  $\abs{dF|_{E}(p)} < 1$, and either $E' = E_\star$ and $\abs{dF|_{E'}(p)} = 1$, or
$\abs{dF|_{E'}(p)} > 1$.
\end{enumerate}
\end{prop}
\begin{proof}
We proceed by induction on $n=d(E_\star,E)$.
Suppose first $n=1$, then $E_\star$ is the unique element at zero distance from $E_\star$, hence (b) obviously holds.
Since $dF|_{E_\star} (p)$ is of modulus one by assumption, Lemma \ref{lem:eigenvalonintersections} implies $\abs{dF|_E(p)}<1$ hence $F|_E$ is hyperbolic.
This shows (a) and (c) hold.

Suppose now the result holds for some $n\ge1 $.
Pick a component $E$ such that $d(E_\star, E)=n+1$.
In particular there exists a component $E'$ such that $d(E_\star,E')=n$ and $p:= E \cap E' \neq \emptyset$.
By the inductive hypothesis, there exists a unique component $E''$ such that $E'' \cap E' = \{p'\} \neq \emptyset$ and $d(E_\star,E'')=n-1$.
Moreover, $\abs{dF|_{E'}(p')} <1$ and since $p$ is also fixed it follows that $\abs{dF|_{E'} (p)} >1$.
By Lemma \ref{lem:eigenvalonintersections} we conclude that $\abs{dF|_E(p)}<1$ hence $F|_{E}$ is hyperbolic.
This proves (a).
Suppose $E$ intersects another component $\widehat{E}$.
The intersection point $\hat{p} = E \cap \widehat{E}$ is fixed and $\abs{dF|_E(\hat{p})}> 1$ so that $\abs{dF|_{\widehat{E}}(\hat{p})}< 1$ by Lemma \ref{lem:eigenvalonintersections}.
In particular $d(E_\star,\widehat{E})$ cannot be $\leq n$ otherwise the inductive hypothesis would imply $\abs{dF|_{\widehat{E}}(\hat{p})}\geq 1$, a contradiction.
This shows (b) and (c).
\end{proof}

\medskip

\noindent {\bf Step 3}. Suppose now that there exists a component $E_\star$ such that $F|_{E_\star}$ has infinite order but is not hyperbolic.

Then $E_\star$ is either an elliptic or a rational curve.
The next lemma excludes the former case.
\begin{lem}\label{lem:ellipticfiniteorder}
If $E_\star$ is complex torus, then  $F|_{E_\star}$ has finite order. 
\end{lem}
\begin{proof}
Denote by $L$ the normal bundle of $E_\star$ in $X$, and by $n<0$ its degree.
Write $E_\star = \C / \Lambda$, and suppose $F|_{E_\star}$ is a translation by $\tau\in \C$.
Any divisor of degree $n$ in $E_\star$ is linearly equivalent to $ (n-1) [0] + [p]$ for some $p \in E_\star$ and such a decomposition is unique in the sense $(n-1) [0] + [p] = (n-1) [0] + [p']$ in the Picard group of $E_\star$ iff $p = p'$.
Now write $L = (n-1) [0] + [p]$ and observe that $F|_{E_\star}$ fixes $L$, whence $n [\tau] + (n-1) [0] + [p] = (n-1) [0] + [p]$.
It follows that $n [\tau] = 0$ in $E_\star$ and $F^n|_{E_\star} = \id$ as required.
\end{proof}

\noindent {\bf Step 4}.
Suppose $X$ is smooth and $W$ has normal crossing singularities but
$\pi:X \to (Y,0)$ is not a nice resolution of $(Y,0)$.
This means that $\pi^*\frm$ admits a non-empty finite set $B$ of base points included in $W$.
Since $f^*\frm = \frm$, we have $F(B)=B$.
By applying Proposition \ref{prop:regularization} to a suitable iterate of $F$, there exists a nice resolution $\mu:\widehat{X} \to X$ dominating $X$ obtained by blowing-up (infinitely near) points above $B$.
Set $\widehat{W}=\mu^{-1}(W)$ and pick $\widehat{F}:\widehat{X} \to \widehat{X}$ a lift of $F$.
Pick any sufficiently large integer $N\ge1$ such that  $F^N$ fixes all the irreducible components of $\widehat{W}$.

We may now apply Steps 1--3 to $(\widehat{X},\widehat{W},F^N)$.
If we are in case (a), the support of $\widehat{W}$ is a chain of rational curves such that $\widehat{F}^N|_{\widehat{E}}$ is hyperbolic for any irreducible component $\widehat{E}$ in $\widehat{W}$. By contracting $\mu^{-1}(B)$ we see that $W$ remains a chain of rational curves, such that $F^N|_E$ is hyperbolic for any irreducible component $E$ in $W$.
Since $F$ fixes any component in $W$ by assumption, we get case (a) of the statement.

Suppose we are in case (b).
There exists a unique component $\widehat{E}_\star$ on which the action of $F$ is not hyperbolic.
If $\widehat{W}$ is the support of a chain of rational curves, we can conclude as before that $W$ is a chain of rational curves.
Moreover, if $\widehat{E}_\star \subset \mu^{-1}(B)$, then $F|_E$ is hyperbolic for any component $E$ in $W$, and we get case (a) of the statement.
If $\widehat{E}_\star$ is not contained in $\mu^{-1}(B)$, then $E_\star=\mu(\widehat{E}_\star)$ is the unique component in $W$ on which $F$ has a non-hyperbolic action.

Suppose $\widehat{W}$ is not the support of a chain of rational curves.
If $\widehat{E}_\star$ is not contained in $\mu^{-1}(B)$, then $\mu$ only contracts rational curves in the chains intersecting $\widehat{E}_\star$.
We argue as before and conclude that $W$ satisfies condition (b) of the statement.

If $\widehat{E}_\star \subseteq \mu^{-1}(B)$, then $\widehat{E}_\star$ is rational, and it intersects at least three components in $\widehat{W}$.
Notice that $W$ has simple normal crossings by assumption, and it is obtained by contracting some rational curves in $\widehat{W}$.
We infer that $W$ is in this case a chain of rational curves, with $F|_E$ hyperbolic for any component $E$ of $W$, hence case (a) of the statement.

\begin{rmk}\label{rmk:otherchains}
Suppose $E_\star$ is a rational curve and $F|_{E_\star}$ has infinite order but is not hyperbolic.
Then $F|_{E_\star}$ is either a translation or a rotation of infinite order.
It admits respectively one or two periodic points.
It follows that in this case the closure of $W \setminus E_\star$ is the union of at most two chains of rational curves, so
that $W$ itself is a chain of rational curves.
\end{rmk}

\subsection{Chains of negative rational curves}

We shall call \emph{negative rational curve} a smooth  rational curve $E$ in $X$ such that $E \cdot E \leq -2$.

\begin{thm}\label{thm:betterstarshaped}
Suppose $(X,W,F)$ is an admissible data, $X$ is smooth, $W$ has simple normal crossings, $F$ fixes any irreducible component of $W$.

Then, up to equivalence of admissible data, we are in one of the following two situations.
\begin{enumerate}
\item[(i)] The support of $W$ is a chain of negative rational curves.
The action of $F|_E$ is hyperbolic for all irreducible components $E$ of $W$, but at most one.
\item[(ii)] There exists a component $E_\star$ with  $E_\star \cdot E_\star \leq -1$ such that $F|_{E_\star}$ has finite order.
In this case, the closure of $W \setminus E_\star$ is a disjoint union of a finitely many chains of negative rational curves, and $F|_E$ is hyperbolic for any component $E$ of $W$ different from $E_\star$.
\end{enumerate}
\end{thm}
Notice that cases (i) and (ii) are not mutually exclusive. The case $E_\star$ is rational, $F|_{E_\star}$ has finite order, and the closure of $W \setminus E_\star$ is the disjoint union of at most two chains of negative rational curves both satisfies (i) and (ii).
\begin{proof}
Suppose first that $W$ is a chain of rational curves.
Since $W$ can be contracted to a point, for any irreducible component $E$ of $W$ we have $E \cdot E \leq -1$.
If there exists $E$ with $E \cdot E=-1$, we can contract this irreducible component and get a shorter chain of rational curves, that can still be contracted to a point.
By induction on the length of the chain, we may suppose that $W$ is a chain of negative rational curves.  Theorem \ref{thm:starshaped} then gives case (i) of the statement. 

If $W$ is not a chain of rational curves then we are in case (b) of Theorem \ref{thm:starshaped} and there exists a unique
component $E_\star$ on which $F$ is not hyperbolic.
By Remark \ref{rmk:otherchains}, we know in fact  that $F|_{E_\star}$ has finite order.
Let $E_1, \ldots, E_n$ be a chain of rational curves in the closure fo $W \setminus E_\star$.
Arguing as before, we can suppose that $E_j \cdot E_j \le -2$ for any $j$.
Since $W$ can be contracted to a point, we infer $E_\star \cdot E_\star \leq -1$.
\end{proof}


\section{Hirzebruch-Jung singularities}

In this  section,we analyze in  detail  the dynamics of an automorphism in a neighborhood of  a chain of rational curves.


\subsection{Dynamics on a chain of rational curves}
Recall that any chain of negative rational curves can be contracted to a cyclic quotient singularity, also called Hirzebruch-Jung singularity.
Any Hirzebruch-Jung singularity $(Y,0)$ is isomorphic to $\C^2$ modulo the action of a finite automorphism of the form
$$
(z,w) \mapsto (\zeta z, \zeta ^q w),
$$
where $\zeta$ is a primitive $m$-th root of unity and $m$ and $q$ are coprime, see~\cite[chapter III.5]{barth-hulek-peters-vanderven:compactcomplexsurfaces}.

We begin with the following observation.
\begin{prop}\label{thm:HJcontraction}
Suppose we are given a chain of negative rational curves in a smooth surface $X$, and an automorphism $F: X \to X$ that leaves the chain invariant.

Then one can contract all curves $E_i$ to a Hirzebruch-Jung singularity $(Y,0)$ and $F$ descends to an automorphism $f:(Y,0) \to (Y,0)$.

Furthermore,  there exist a local analytic diffeomorphism $\wt{f}: (\C^2,0) \to (\C^2,0)$ and a linear automorphism $\gamma$ of $\C^2$ of finite order such that the quotient $\C^2/\langle \gamma \rangle$ is isomorphic to $Y$, and $\wt{f}$ descends to an automorphism of $Y$ that is analytically conjugated to $f$.
\end{prop}
\begin{proof}
Contract the chain of rational curves to a Hirzebruch-Jung singularity $(Y,0)$.
Since $F$ is an automorphism and preserves $\bigcup E_i$, it induces an automorphism $f$ on $Y \setminus \{0\}$ which extends to the singularity by normality.

Write $(Y,0)$ as a quotient of $\C^2$ by a finite order automorphism.
Restricting $Y$ if necessary, we may assume that the natural projection $\mu:(\C^2, 0)\to (Y,0)$ induces an unramified covering from $B\setminus\{0\}$ onto $Y\setminus \{0\}$ where $B$ is a small ball centered at $0$.
Since $B\setminus\{0\}$ is simply connected, it follows that $f:(Y,0) \rightarrow (Y,0)$ lifts to $\wt{f}:B\setminus\{0\} \rightarrow B\setminus\{0\}$.
Finally $\wt{f}$ extends through $0$ by Hartogs' Lemma.
\end{proof}


\subsection{Local conjugacy along curves of fixed points}
We shall also need the following refinement.
\begin{thm}\label{thm:HJcurvefixedpoints}
Suppose $(X,W,F)$ is an admissible data, $X$ is smooth, $W$ has simple normal crossings, and $F: X \to X$ an automorphism.

Assume $E_1, \ldots, E_n$ is a chain of negative rational curves in $W$ such that $F(E_i) = E_i$ and $F|_{E_i}$ is hyperbolic for all $i$.
Assume moreover that there exists a component $E_\star$ in $W$ fixed pointwise by $F$ and intersecting transversely $E_1$.

Then one may contract  $\bigcup_i E_i$ to a Hirzebruch-Jung singularity $(Y,0)$, and $F$ induces an automorphism $f:(Y,0) \to (Y,0)$.
Moreover,  there exist coordinates $(z,w) \in \C^2$, a finite-order automorphism $\gamma(z,w) = (\zeta z, \zeta^q w)$ with $\zeta$ a primitive $m$-th root of unity, and $\gcd \{m,q\} = 1$ such that $(Y,0)$ is isomorphic to $\C^2/\langle \gamma \rangle$.
And $f$ lifts to a linear map $(z,w) \mapsto  (z, \alpha w)$ for some $0<\abs{\alpha}<1$.
\end{thm}
\begin{proof}
By Proposition~\ref{thm:HJcontraction} there exists a linear automorphism $\gamma(z,w) = (\zeta z, \zeta^q w)$ such that $\C^2/\langle \gamma \rangle$ is isomorphic to the Hirzebruch-Jung singularity $(Y,0)$ obtained by contracting the chain of negative rational curves $E_1, \ldots, E_n$.
Moreover we can lift $f:(Y,0) \to (Y,0)$ to an automorphism $\wt{f}: (\C^2,0) \to (\C^2,0)$ such that $\wt{f} \circ \gamma=\gamma^k \circ \wt{f}$ for a suitable $k \in \N$ with $\gcd\{k,m\}=1$.
Set $\wt{E}=\mu^{-1}(\pi(E_\star))$, where $\mu$ is the canonical projection $\C^2 \to \C^2/\langle \gamma \rangle$, and $\pi:X \rightarrow Y$ is the contraction map of $E_1, \ldots, E_n$.

Since $E_\star$ is a curve of fixed points by $F$, we get that $\wt{E}$ is a (possibly singular reducible) curve of fixed points for $\wt{f}^m$.
Replacing $\wt{f}$ by  $\gamma^l \circ \wt{f}$ for a suitable $l$ if necessary, we can suppose that $\wt{E}$ is a curve of fixed points for $\wt{f}$.
It follows that $1$ is an eigenvalue for $d\wt{f}(0)$.
Take a point $\wt{p} \in \wt{E} \setminus \{0\}$.
Since $\mu$ is an unramified covering outside $0$ and $\pi$ is a biholomorphism on $E_\star \setminus \bigcup_j E_j$, we infer $\det{d\wt{f}(\wt{p})}=\det{dF(p)}=\alpha$ with $0 < \abs{\alpha} < 1$, and $p=\pi^{-1}(\mu(\wt{p}))$.
By letting $\wt{p}$ tend to $0$, we get that the eigenvalues of $d\wt{f}_0$ are $1$ and $\alpha$.
It follows that $\wt{E}$ is the central manifold of $\wt{f}$ at $0$, and hence it is smooth.
The condition $\wt{f} \circ \gamma= \gamma^k \circ \wt{f}$ restricted to $\wt{E}$ implies $k=1$.
We denote by $\wt{D}$ the stable manifold at $0$, that is smooth and transverse to $\wt{E}$ at $0$.

Pick $\phi,\psi \in \frm$ such that $\wt{D}=\{\phi(z,w)=0\}$ and $\wt{E}=\{\psi(z,w)=0\}$.
Denote by $\phi_1=az+bw$ (resp., $\psi_1$) the linear part of $\phi$ (resp., $\psi$).
Since $\wt{D}$ is invariant by the action of $\gamma$, we infer that $a\zeta z + b \zeta^q w$ is proportional to $az + bw$.
If $q\neq 1$, we deduce that either $a=0$ or $b=0$.
If $q=1$, then $\gamma$ is a homothety, and we can suppose $a=0$ or $b=0$ up to a linear change of coordinates.
We can argue analogously for $\wt{E}$.
It follows that we can assume $\phi_1=z$ and $\psi_1=w$.

By direct computation, we get $\phi \circ \gamma = \zeta \phi$ and $\psi \circ \gamma = \zeta^q \psi$.
Then $\Phi(z,w)=(\phi(z,w), \psi(z,w))$ defines an automorphism on $(\C^2,0)$ such that $\Phi \circ \gamma = \gamma \circ \Phi$.
After conjugating by $\Phi$, we get $\wt{E}=\{z=0\}$, $\wt{D}=\{w=0\}$, and $\gamma(z,w)=(\zeta z, \zeta^q w)$.

The stable manifolds at any point $\wt{p}\in \wt{E}$ defines a $\gamma$-invariant holomorphic foliation $\wt{\mc{F}}$.
Up to a $\gamma$-equivariant change of coordinates, we can then suppose that $\wt{\mc{F}}$ is given by $\{z=\text{const}\}$.
In these coordinates, $\wt{f}$ is given by
$$
\wt{f}(z,w)=\big(z, \alpha w (1+\eps(z,w))\big),
$$
where $w$ divides $\eps(z,w)$.
Since $\wt{f}$ is $\gamma$-invariant, we infer $\eps \circ \gamma = \eps$.
For any $z$, the Koenigs theorem (see for example \cite[section 6.1]{milnor:dyn1cplxvar}) produces a holomorphic invertible germ $\eta_z:(\C,0)\to (\C,0)$ that conjugates $w \circ \wt{f}(z,w)$ to $w \mapsto \alpha w$.
The change of coordinates $\eta_z$ is given by
$$
\eta_z(w):=\eta(z,w):= w \prod_{n=0}^\infty \big(1+\eps \circ \wt{f}^n(z,w)\big),
$$
so that
$$
\eta \circ \gamma(z,w)
= \zeta^q w \prod_{n=0}^\infty \big(1+\eps \circ \wt{f}^n \circ \gamma(z,w)\big)
= \zeta^q w \prod_{n=0}^\infty \big(1+\eps \circ \wt{f}^n (z,w)\big)
= \zeta^q \eta(z,w).
$$
We conclude by conjugating $\tilde{f}$ by the map $(z,w) \mapsto (z, \eta_z(w))$, that is $\gamma$-invariant.
\end{proof}


\section{Orbifold Structures}\label{sec:orbifold}

In this section, we recall some properties of orbifold structures on compact Riemann surfaces
and the notion of orbibundle. References include \cite{scott:geometries3mflds}  or \cite{satake:gaussbonnetVmflds,furuta-steer:seifertfibredhomology} where orbifolds are referred to as $V$-manifolds. 


\subsection{Orbifolds}

An orbifold structure on a  Riemann surface $S$ is the data of a finite collection of
points $\sum m_i p_i$ with multiplicities $m_i \in \N^*$. The multiplicity $\mult(p)$ of any point $p$ in $S$ is defined to be $m_i$ if $p = p_i$ and $1$ otherwise. To simplify  we shall talk about orbifold in place of orbifold structure on a compact Riemann surface.

An orbifold chart at a point $p$ with multiplicity $n$ is an effective holomorphic action of $\Z/n\Z$ on the unit disk $\D$ fixing only the origin, and a holomorphic map of $\D$ to a neighborhood $U$ of $p$ in $S$ that factors through
$\D / (\Z/n\Z)$ as an isomorphism. 

\begin{eg}
Pick any coprime integers $p,q \ge 1$.
The weighted projective space $\P^1(p,q)$ is defined as the quotient of $\C^2 \setminus \{ (0,0)\}$ modulo the action of $\C^*$
given by $t\cdot (x,y) = (t^p x, t^q y)$. The chart $\{x=1\}$ is isomorphic to $ \C / \langle \zeta_p^q \rangle$ for a primitive $p$-th root of unity $\zeta_p$, whereas  the chart $\{y=1\}$ can be analogously identified  to $ \C / \langle \zeta_q^p \rangle$. It follows that $\P^1(p,q)$
 is diffeomorphic to the Riemann sphere, and
 carries a natural orbifold structure where $\mult ([x:y]) =1$ if $xy \neq 0$, $\mult ([1:0]) = p$ and $\mult ([0:1]) = q$.
\end{eg}

An orbifold  map  $\bar{f}: (S, n) \to (S', n')$  is the data of a holomorphic map $f: S \to S'$ that lifts locally to orbifold charts. In other words, it is a holomorphic map $f$ such that $\mult (p ) \deg (f,p)$ is a multiple of $\mult (f(p))$  for any $p\in S$.
When equality holds $\mult (p ) \deg (f,p) = \mult (f(p))$ for all $p$, then the orbifold map is said to be unramified.
An  orbifold covering map is an orbifold map that is proper and unramified.

Write $\Aff$ for the group of affine transformations of the complex plane, and $\bb{H}$ for the upper-half plane.
Observe that any discrete group of $\Aff$ (resp. $\PSL(2,\R)$) acting properly discontinuously  on $\C$ (resp. on $\bb{H}$)
defines a natural orbifold structure on the quotient space $\C /G$ (resp. $\bb{H} /G$). The multiplicity of a given point is 
then equal to the order of the isotropy group of one (or any) of its preimage in $\C$ (resp. in $\bb{H}$).

\begin{thm}\label{thm:class-orbifold}
Suppose $S$ is an  orbifold whose underlying topological space is compact. Then we are in one of the following four (exclusive) situations:
\begin{enumerate}
\item $S$ is a weighted projective space;
\item there exists a finite group $G$ of $\PGL (2, \C)$ such that $S$ is isomorphic to $\P^1 / G$;
\item there exists a discrete subgroup $G$ of $\Aff (\C)$ acting cocompactly on $\C$ such that $S$ is isomorphic to $\C / G$;
\item there exists a discrete subgroup $G$ of $\PSL (2, \R)$ acting cocompactly on $\bb{H}$ 
such that $S$ is isomorphic to $\bb{H} / G$.
\end{enumerate}
\end{thm}
We refer to~\cite[Theorems 2.3 and 2.4]{scott:geometries3mflds} for proofs, see also~\cite[Theorems 1.1 and 1.2]{furuta-steer:seifertfibredhomology}.
It is of common use to say that an orbifold is \emph{good} when it falls in one of the three cases (b--d).

\smallskip

By Selberg's lemma any finitely generated subgroup of  $\GL(n,\C)$ admits a normal torsion-free subgroup, see~\cite[Lemma 8]{selberg:discgroupssymspaces} or~\cite{bundgaard-nielsen:normalsubgroupsfiniteindex,fox:fenchelconjFgroups} for $n=2$. It follows that any good orbifold admits a finite orbifold covering by a genuine Riemann surface (i.e. an orbifold with $\mult(p) =1$ for all $p$).

\begin{cor}\label{cor:orbifoldglobalquotients}
Suppose $S$ is a compact good orbifold.
Then there exists a compact Riemann surface $\wt{S}$ and a finite subgroup $G$ of $\aut(\wt{S})$ such that $S$ is isomorphic to $\wt{S} / G$.
\end{cor}

We refer to~\cite[Theorem 2.5]{scott:geometries3mflds} for a proof.


\subsection{Orbibundles}

An orbibundle on an orbifold $S$ is a complex analytic space $L$ equipped with a map 
$\pi: L \to S$ such that for any $p\in S$ of multiplicity $m\ge 1$ there exists an integer $q$, a primitive $m$-th root of unity $\zeta$, and a neighborhood $U$ of $p$ in $S$ such that $\D \to \D /\langle \zeta z \rangle \simeq U$ is an orbifold chart, and there exists an analytic isomorphism $\pi^{-1}(U) \simeq \D \times \C \mod (\zeta z , \zeta^q w)$ such that the diagram commutes:
$$
\xymatrix{
(z,w) \in \D \times \C \ar[rr]^{ \,\, \mod  (\zeta z , \zeta^q w)} \ar[d]
& & \pi^{-1}(U)
\ar[d]
\\
z \in \D
\ar[rr]^{ \mod  (\zeta z)}
& & U}
$$
Coordinates $(z,w)$ on $\D \times \C$ as above are called an orbifold trivialization of the
orbibundle at $p$.

Orbibundles are also known as line $V$-bundles, see e.g. \cite{satake:gaussbonnetVmflds}.

Observe that $L$ may have (cyclic quotient) singularities whereas $S$ is always smooth. 
In other words, an orbifold line bundle needs not be a locally trivial fibration over an orbifold point.

Pick $L \to S$ any line bundle on a compact Riemann surface, and suppose $G$ is a finite group
acting linearly on $L$. This means that $G$ acts  
by linear transformations on the fibers. One may then form the quotient space $L/G$, and one checks that  
this space has a natural structure of orbibundle over the orbifold $S/G$.
Conversely, one has
\begin{thm}
\label{thm:class-orbiline}
Let $L\to S$ be any orbibundle on a compact good orbifold.

Then there exists a holomorphic line bundle $L' \to S'$ over a (genuine)  compact Riemann surface $S'$ and a finite group $G$ acting linarly on $L'$  such that $L$ is isomorphic to $L'$ quotiented by the action of $G$.
\end{thm}
We refer to~\cite[Theorem 1.3]{furuta-steer:seifertfibredhomology} and \cite[Section 2]{ross-thomas:weightedprojembeddingorbifoldscscKahler} and references therein.


\section{The Main Theorem}

In this section we prove a more precise version of  Theorem~\ref{thm:main} classifying
contracting automorphisms of a complex normal surface singularity.
Before stating our main result precisely we describe some examples.

\begin{eg}\label{nrm:HJ}
Pick any two coprime integers $m,q \ge 1$, and $\zeta$ a primitive $m$-th root of unity.
Denote by $\gamma$ the automorphism of $\C^2$ of order $m$ defined by $\gamma(z,w)=(\zeta z, \zeta^q w)$.
Suppose $\wt{f}$ is an automorphism of $\C^2$ of one of the following forms:
\begin{enumerate}
\item $\wt{f}(z,w)=(\alpha z , \beta w )$,
with $\alpha, \beta \in \C$, $0 < \abs{\beta} \leq \abs{\alpha} < 1$;
\item $\wt{f}(z,w)=(\alpha z , \alpha^u  w + z^u)$,
with $0 <  \abs{\alpha} < 1$, $u \in \N^*$ and $q \equiv u \mod m$;
\item $\wt{f}(z,w)=(\beta w , \alpha z)$, with $\alpha, \beta \in \C$, $0 < \abs{\alpha\beta} < 1$, and $q^2 \equiv 1 \mod m$.
\end{enumerate}
The automorphism $\wt{f}$ then descends to a contracting automorphism of the Hirzebruch-Jung singularity
$(Y,0)=(\C^2,0)/\langle \gamma \rangle$.
\end{eg}

\begin{eg}\label{nrm:2}
Let $L \to S$ be a holomorphic line bundle of negative degree on a Riemann surface, and denote by
$F_\alpha : L \to L$ the bundle map such that the restriction of $F_\alpha$ to each fiber is conjugated to 
$w \mapsto \alpha w$ for some $| \alpha| < 1$. The contraction of the zero section of $L$
gives rise to a normal cone singularity $(Y,0)$, and $F_\alpha$ induces a contracting automorphism 
$f: (Y,0)\to (Y,0)$ since $| \alpha| < 1$. 

Observe that $Y$ is the normalization of the affine space $\textup{Spec} \left( \bigoplus_{n\ge 0}  H^0( Y, L^{\otimes -n})\right)$. It is thus endowed with a natural proper $\C^*$-action, and $f$ belongs to the flow induced by this action. 
\end{eg}

\begin{eg}\label{nrm:3}
Let $L \to S$ be a holomorphic line bundle of negative degree on a Riemann surface, and choose
$\phi: S \to S$ an automorphism of finite order such that $\phi^* L$ is isomorphic to $L$.
Fix such a bundle isomorphism $\Phi: L \to \phi^* L $. The composite map $F$ on the total space of $L$
described in the following commutative diagram
$$
\xymatrix{
L
\ar@/^2pc/[rrr]^F
\ar[r]^{\Phi}
\ar[d]
&
\phi^*L
\ar[r]
\ar[d]
&
L
\ar[r]^{F_\alpha}
\ar[d]
&
L
\ar[d]
\\
S
\ar[r]^{\id}
&
S
\ar[r]^{\phi}
&
S
\ar[r]^{\id}
&
S
}
$$
 descends to $(Y,0)$ and induces a contracting automorphism if $|\alpha|$ is sufficiently small.
\end{eg}

\begin{eg}\label{nrm:bundle}
In the same situation as in the previous example, 
suppose moreover that we are given a finite group $G$
acting by automorphisms on $S$ and linearly on $L$, whose action commutes with $F$. 
Then $F$ descends to the quotient space $Y/G$ as a contracting automorphism. The singularity 
$Y/G$ is again weighted homogeneous since it carries a proper $\C^*$-action.

Observe that $L/G \to S/G$ is an orbibundle. Conversely by Theorem~\ref{thm:class-orbiline}  any orbi-bundle 
$L' \to S'$ of negative degree on a good orbifold arises as a quotient of a genuine holomorphic line bundle.
In particular, one may contract the zero section of $L'$ to a weighted singularity. Since $L'$ carries a natural proper $\C^*$-action,  this singularity also supports contracting automorphisms. 
\end{eg}

\begin{thm}\label{thm:class-admissibledata}
Any contracting automorphism of a complex normal surface singularity is obtained by one of the constructions explained in Examples~\ref{nrm:HJ} or ~\ref{nrm:bundle} above.
\end{thm}

Observe that this result implies Theorem~A.

\begin{proof}
Start with $Y$ a complex surface having a (unique) isolated normal singularity at $0$, and $f : Y \to Y$
an automorphism fixing $0$ such that $f(Y)$ is relatively compact in $Y$ and $\bigcap_{n=0}^\infty f^n(Y)=\{0\}$.
We fix a resolution of singularities $\pi : X \to Y$, such that $W:= \pi^{-1}(0)$ has simple normal crossings and $f$ lifts to a holomorphic map $F: X \to X$. The existence of such $\pi$ is guaranteed by Proposition~\ref{prop:regularization}.

Apply~Theorem \ref{thm:betterstarshaped} (to an iterate $F^N$ that fixes any irreducible component of $W$).
Up to equivalence of admissible data, we have two possibilities.

\medskip
\noindent {\bf Case 1}.
The exceptional locus $W$ is a chain of rational curves. By Proposition~\ref{thm:HJcontraction}  $(Y,0)$ is isomorphic to $(\C^2,0)/\langle \gamma \rangle$ for some automorphism $\gamma$ of finite order.
Denote by $\mu:(\C^2,0) \to (Y,0)$ the natural projection.
Then $f$ lifts to a local automorphism  $\wt{f}:(\C^2,0) \to (\C^2,0)$ such that $\mu \circ \wt{f} = f \circ \mu$ and  $\wt{f} \circ \gamma = \gamma^k \circ \wt{f}$ for a suitable $k \in \N$.
Observe that for any open neighborhood $U$ of the origin,  $\wt{f}^N(U)$ is relatively compact in $U$ for $N$ sufficiently large,
hence $\wt{f}$ is a contracting automorphism.
We now conclude  in this case by giving the classification of those attracting germs commuting with the group generated by $\gamma$.
\begin{prop}
Let $\gamma:(\C^2,0) \rightarrow (\C^2, 0)$ be an automorphism of finite order, and $\wt{f}:(\C^2,0) \rightarrow (\C^2, 0)$ a contracting automorphism satisfying  $\wt{f} \circ \gamma = \gamma^k \circ \wt{f}$ for some $k \in \N$. Then in suitable holomorphic coordinates
we are in the situation of Example~\ref{nrm:HJ}.
\end{prop}
\begin{proof}
By a Theorem of Cartan we  may suppose $\gamma$ is linear, and write it under the form $\gamma(x,y)=(\zeta x, \zeta^q y)$, where $\zeta$ is a primitive $m$-th root of unity ($m\ge2$), and $\gcd\{m,q\}=1$.
Observe that we may choose $k$ prime with $m$.
If we write
$$
d\wt{f}(0)=
\left(
\begin{array}{cc}
a & b \\
c & d
\end{array}
\right)$$
then a direct computation shows
$$
d\wt{f}(0) \circ \gamma = 
\left(
\begin{array}{cc}
\zeta a & \zeta^q b \\
 \zeta c & \zeta^q d
\end{array}
\right) \text{ and }
\gamma^k \circ d\wt{f}(0)
=
\left(
\begin{array}{cc}
 \zeta^k a &  \zeta^k b \\
\zeta^{qk} c & \zeta^{qk} d
\end{array}
\right),
$$
and since $d\wt{f}(0) \circ \gamma = \gamma^k \circ d\wt{f}(0)$ we infer that
\begin{enumerate}
\item either $k \equiv q \equiv 1 \mod m$;
\item or $k \equiv 1  \mod m$, $ q\not\equiv 1 \mod m$, and $b=c=0$;
\item or $k \not\equiv 1 \mod m$,  $k \equiv q \mod m$, $q^2 \equiv 1 \mod m$, and $a=d=0$.
\end{enumerate}

In cases (a) and (b) we have $k \equiv 1 \mod m$, so that we can  directly use Theorem~\ref{thm:PDcommutingGammadiag}.
In case (c),  the eigenvalues of $d\wt{f}(0)$ are opposite, hence there is no resonance, and
the Poincar\'e-Dulac normal form is linear. By Remark~\ref{rmk:invariantconjugation} $\wt{f}$ can be conjugated to a linear map
$(\alpha w, \beta z)$ by a change of coordinates that preserves the cyclic group generated by $\gamma$. This concludes the proof.
\end{proof}

\medskip
\noindent {\bf Case 2}.
Suppose $W$ is not a chain of rational curves.
By Theorem \ref{thm:betterstarshaped}, there exists an irreducible component $E_\star$ of $W$ such that $F|_{E_\star}$ has finite order, and the closure of $W \setminus E_\star$ is a finite number of chains of negative rational curves.

\smallskip

We assume first that $F|_{E_\star}=\id$.

By contracting all chains of rational curves, we obtain an admissible data $(X', E'_\star, F')$, where $F'|_{E'_\star} = \id$.
Observe that $X'$ now admits (Hirzebruch-Jung) singularities $p_1, \ldots, p_r \in E'_\star$ each of which is an  image of a contracted chain.

To understand the local situation at a point $p\in E'_\star$ we  apply Theorem~\ref{thm:HJcurvefixedpoints}. 
\begin{enumerate}
\item
If $p\neq p_j$ for all $j$, then $X'$ is smooth at $p$ and there is neighborhood $V_p$ of $p$ in $X'$ and an analytic diffeomorphism $\Phi_p : \D \times \D\to V_p$ sending $\{ w =0 \}$ to  $E'_\star$ and conjugating $F'$ to $(z,w) \mapsto (z , \alpha w)$ for some $|\alpha| <1$.
\item
When $p =p_j$ is a cyclic quotient singularity, we may find  a neighborhood $V_p$ of $p$ and a finite Galois cover 
$\Phi_p : \D \times \D \to  V_p$ with Galois group generated by
$(z,w) \mapsto (\zeta_j z, \zeta_j^{q_j} w)$ with $\zeta_j$ a primitive $m_j$-th root of unity, and $q_j$ prime with $m_j$. In this situation $E'_\star$ is the image under the natural projection map of $\{ w=0\}$ and as before $F'$ lifts to an automorphism $(z,w) \mapsto (z , \alpha w)$ for some $|\alpha| <1$.
\end{enumerate}
We may (and shall) assume that for any $p\neq q$ the intersection $V_p \cap V_q \cap E'_\star$ is simply connected and does not contain
any of the points $p_j$. It follows that the map $\Phi_{q}^{-1}\circ \Phi_p$ is always well defined in a neighborhood of 
$\Phi_p^{-1} ( V_p \cap V_q \cap E'_\star)$  and is unique up to a (pre- or post-)composition by a finite order linear automorphism.

We also observe that in both cases the foliation induced by $\{ z = \text{const}\}$ is intrinsically defined, since it coincides with the stable foliation of $F$, and  leaves are stable complex curves of points lying in $E'_\star$. It follows that  the transition map $\Phi_{q} \circ \Phi_p^{-1}$ for any two points $p\neq q$ can be written under the form 
$\Phi_{q}^{-1} \circ \Phi_p (z,w) = (\Phi_{q,p}(z), \star )$. Since the restriction of this map to any vertical line 
$\{ z = \text{cst}\}$ commutes with $w \mapsto \alpha w$ we actually get  
\begin{equation}\label{e:patch}
\Phi_{q}^{-1} \circ \Phi_p (z,w) = (\Phi_{q,p}(z), H_{q,p}(z) \cdot w)~,
\end{equation}
for some nowhere vanishing holomorphic map $H_{q,p}$.

To interpret geometrically what is going on, we note that $E'_\star$ is smooth, and we endow it with the orbifold structure
such that  $\mult (p) = m_j$ if $p= p_j$ as in the discussion (b) above, and $\mult(p) =1$ otherwise. We may then construct
an analytic space $L$ by patching together $\{ (V_p \cap E'_\star) \times \C\}$ for all $p\in E'_\star$ using the transition maps given by~\eqref{e:patch}. The natural projection map $L \to E'_\star$ makes $L$ into an orbibundle and we have a natural embedding $X' \to L$ such that $E'_\star$ appears as the zero section of the orbibundle. By construction $F'$ acts linearly on each fiber of $L$.

We now claim that the orbifold structure on $E'_\star$ is good. Indeed if this is not the case, then $E'_\star$ is a rational curve 
with one or two multiple points and it follows that $(Y,0)$ admits a resolution of singularities such that the exceptional divisor is a chain of rational curves. We thus fall into Case (1) which has been already treated. 

Since $E'_\star$ is a good orbifold, the orbibundle $L\to E'_\star$ is a global quotient of a genuine holomorphic line bundle on a compact Riemann surface by a finite group acting linearly on the fiber, see Theorem~\ref{thm:class-orbiline}. Observe that $F'$ lifts to the line bundle since it acts linearly on the fibers of $L$, so that we are in the situation of the Example~\ref{nrm:bundle}.

Finally we go back to the original setting, and suppose $F|_E$ is not necessarily the identity but has finite order equal to $N\ge2$.
We argue as above with the map $F^N$ thereby conjugating this map to a linear map acting on an orbibundle $L \to E'_\star$
by multiplication by some $\alpha$ with $|\alpha| <1$.
Since the fibers of $L$ are also the stable manifolds of the periodic points of $F$ (i.e. of the points lying in $E'_\star$), it follows $F$ preserves these fibers.

We claim that $F$ extends as an automorphism of the total space of the orbibundle $L$ acting linearly on the fibers. 
To see this recall that $X'$ is a neighborhood of the zero section of $L$ and the closure of $F(X')$ is relatively compact in $X'$. 
It follows that $X' \setminus F(X')$ is a fundamental domain for the action of $F$ on $X'\setminus E'_\star$.
We may thus take an infinite number of copies $X_i$ of $X' \setminus F(X')$ indexed by $i\in \N^*$ and using $F$ patch 
the "inner" boundary corresponding to $\partial F(X')$ in $X_i$ to the "outer" boundary  corresponding to $\partial X'$ in $X_{i-1}$.
By adding $X_0=F(X')$ to $X_1$ we obtain an analytic space $\mathfrak{X}$  that contains naturally
$X'\cong X_1$ and is endowed with an automorphism $\mathfrak{F} :\mathfrak{X} \to \mathfrak{X}$, given as the shift $\on{id}:X_i \to X_{i-1}$ for $i \geq 2$, and whose restriction to $X'$ equals $F$.

Since $\mathfrak{F} (X_i) = X_{i-1}$ for $i \in \N^*$, for any $p\in \mathfrak{X}$ the point $\mathfrak{F}^n(p)$ eventually belongs to $X'$.
Since the stable foliation in $X'$ is $F'$-invariant, it can be extended to an $\mathfrak{F}$-invariant holomorphic foliation to $\mathfrak{X}$.
The intersection of any leaf $\mathcal{L}$ of $\mathcal{F}$ with $X_i$ is an annulus $A_i$. Since $F^N = \id$, the annuli $A_i$ and $A_{i+N}$ are analytically diffeomorphic, which implies $\mathcal{L}$ to be isomorphic to $\C$. Observe that there exists a unique isomorphism $\mathcal{L} \simeq \C$ sending $\mathcal{L}\cap E'_\star$ to $0$ up to a homothety, and $\mathfrak{F}$ acts as a linear transformation from a leaf $\mathcal{L}$ to its image.

Denote by $F_\alpha$ the map on the total space of $L$ obtained by multiplication by $\alpha$ on each fiber. 
Using the two natural embeddings of $X'$ in $\mathfrak{X}$ and $L$, we can construct
a natural  map $\Phi: \mathfrak{X} \to L$ defined by $\Phi (p) := F_\alpha^{-k} \circ \mathfrak{F}^{Nk}(p)$ for any $k$ large enough.
It is not difficult to check that $\Phi$ is an isomorphism sending $\mathcal{F}$ to the fibration $L\to E'_\star$ which conjugates $\mathfrak{F}$ to $F$ in the neighborhood of the zero section of $L$.
We have thus proved that $F$ extends to an automorphism of the total space of $L$ acting linearly on the fibers as required.

As above, we may assume that $E'_\star$ is a good orbifold isomorphic to a quotient of a  finite group $G$ acting freely on a Riemann surface $S$. Moreover we may suppose there exists a holomorphic line bundle $L_S \to S$ and a linear action of $G$ on $L$ lifting the one on $S$ such that $L$ is isomorphic to $L_S /G$. 
Since $S \to E'_\star$ is a Galois unramified cover (in the sense of orbifold), the finite order automorphism $F|_{E'_\star}$ lifts to a finite order automorphism $F_S : S \to S$. And since $F$ acts linearly on $L$ and $F|_{E'_\star}$ can be lifted to $S$, the map $F$ also lifts to $L_S$ and acts linearly on the fibers.
This finally proves that we are in the situation of the Example~\ref{nrm:bundle}.
\end{proof}

\section{The orbit space of a contracting automorphism}

In this section, we prove  Corollary~\ref{cor:mainsurfaces}. Recall that the orbit space 
$S(f) := (Y \setminus \{ 0 \} )/ \langle f \rangle$ of a contracting automorphism
$f : (Y,0) \to (Y,0)$ is a compact complex surface. Our aim is to describe its geometry.

To do so we rely on the next observation whose proof is left to the reader.

\begin{lem}\label{lem:coverings}
Suppose $f : (Y,0) \to (Y,0)$ is a contracting automorphism of a complex normal surface singularity.
For any integer $N\ge1$,  the natural map   $S(f^N) \to S(f)$ is an $N$-cyclic unramified (Galois) covering.
\end{lem}
A generator of the Galois group of the covering is given by the map induced by $f$ on $S(f^N)$.

\medskip

\begin{proof}[Proof of Corollary \ref{cor:mainsurfaces}]
We pick a contracting automorphism $f : (Y,0) \to (Y,0)$ of a complex normal surface singularity and we apply
Theorem~\ref{thm:class-admissibledata}.

Suppose we are in the situation described in the Example~\ref{nrm:HJ}.
In this case, $(Y,0)$ is a cyclic quotient singularity.
We can thus write $Y = \C^2 / \Gamma $ for some finite subgroup of $\textup{GL}(2,\C)$ acting freely on $\C^2 \setminus \{ 0 \}$,
and $f$ lifts to a polynomial automorphism $\wt{f}: \C^2 \to \C^2$ such that $\wt{f}^n (p) \to 0$ for all $p\in \C^2$.
It follows that the surface $S(\wt{f})$ is a primary Hopf surface, see~\cite[IV.18]{barth-hulek-peters-vanderven:compactcomplexsurfaces}.
The natural projection map $\C^2 \setminus \{0\} \to Y \setminus \{0\}$ induces a cyclic Galois covering $S(\wt{f}) \to S(f)$ and $S(f)$ is a secondary Hopf surface.
The universal cover of these surfaces is isomorphic to $\C^2 \setminus \{ 0 \}$, and $b_1 (S(f)) = 1$ which implies them to be non-K\"ahler.
This proves (a) and (b) in this case.

\smallskip

Suppose now that $Y$ is  obtained as the contraction of the zero section of a line bundle $L \to S$  of negative degree over a compact Riemann surface, that is we are in  Example~\ref{nrm:2}.
Then a suitable iterate of $f$ is induced by the bundle automorphism $F: L \to L$ acting by multiplication by some $|\alpha| <1$  on each fiber of $L$.
If $E$ denotes the zero section of $L$, then the orbit space $S(f)$ is isomorphic to the quotient space $(L  \setminus E) / \langle F \rangle$.
The latter space is an principal elliptic fibre bundle\footnote{In fact any principal elliptic fibre bundle arises as follows (with $L$ possibly of non-negative degree), see~\cite[Proposition V.5.2]{barth-hulek-peters-vanderven:compactcomplexsurfaces}.} over $E$.

Since $L$ has negative degree, its dual $L^{-1}$ can be endowed with a hermitian metric $h$ whose curvature is a smooth closed positive $(1,1)$-form $\om$ on $E$.
In a local trivialization of $L$ over an open subset $z \in U \subset E$, one can write $|\cdot |_h = e^{ -\varphi (z)} \, | \cdot |$ and $\om = dd^c \varphi$ hence $\varphi$ is psh.
Over $U$, the dual metric $h^*$ on $L$ writes as follows $|\cdot |_{h^*} = e^{ \varphi (z)} \, | \cdot |$.
We may thus define a psh function $\psi (z,w) := \log |w|_{h^*}$ on the total space of $L$ which satisfies the relation $\psi \circ F = \psi + \log | \alpha|$.
We conclude that $S(f)$ carries a positive closed $(1,1)$ current $T :=  dd^c \psi$ which is exact $T = d ( d^c \psi)$ whence $S(f)$ cannot be K\"ahler.
This proves (a) and (b) in this case too. 

\smallskip

In the case of Example~\ref{nrm:3}, $Y$ is obtained as before by contracting the zero section of a genuine line bundle over a Riemann surface but only an iterate $f^N$ of $f$ is induced by a bundle automorphism.
However $f$ lifts to a map on $L$ that preserves the fibers.
From Lemma~\ref{lem:coverings} we conclude that $S(f^N) \to S(f)$ is a cyclic Galois cover preserving the elliptic fibration as required.

\smallskip

Finally in the general case, we know that there exists a finite group $G$ acting linearly on a line bundle $L \to S$ of negative degree as before, that $Y \setminus  \{0 \}$ is the quotient of $L \setminus E$ by $G$, and that an iterate $f^N$ of $f$ lifts to a linear automorphism of $L$.
It follows that $S(f^N)$ (hence $S(f)$ by Lemma~\ref{lem:coverings}) admits a Galois unramified covering from an principal elliptic fiber bundle. 
This concludes the proof of (a) and (b).

\smallskip

We now prove (c). If $(Y,0)$ is  a cyclic quotient singularity, then we have seen that $S(f)$ is a Hopf surface
hence $\kod(S(f)) = -\infty$.
Suppose $S(f)$ is a principal elliptic bundle over a base $E$.
If the base is a rational curve, then~\cite[Theorem V.5.4]{barth-hulek-peters-vanderven:compactcomplexsurfaces} implies $S(f)$ to be a Hopf surface.
If the genus of $E$ is greater or equal to $2$, then $\kod(S(f)) =1$ by~\cite[Proposition V.12.5 (ii)]{barth-hulek-peters-vanderven:compactcomplexsurfaces}.
Finally suppose $E$ is an elliptic curve.
By \cite[\S V.5]{barth-hulek-peters-vanderven:compactcomplexsurfaces}, $S(f)$ is either a Kodaira surface or a complex $2$-dimensional torus.
Since $S(f)$ is not K\"ahler, the latter case cannot appear.

\smallskip

In general, $S(f)$ admits a finite Galois unramified cover from a principal elliptic bundle.
We conclude the proof noting that the Kodaira dimension is preserved under finite unramified covers, and that 
both classes of Hopf and Kodaira surfaces are also stable under finite unramified covers.
\end{proof}

\appendix

\section{Poincar\'e-Dulac normal forms for invariant germs}

Our aim is to given normal forms for contracting automorphisms of $ (\C^d,0)$ commuting with a finite group action. 
Our result is a discrete analog of the results obtained by Sanchez-Bringas~\cite{sanchezbringas:normalformsinvariantvectorfields} in the case of holomorphic vector fields.

Let $f:(\C^d,0)\rightarrow (\C^d,0)$ be an attracting fixed point germ, and $\Gamma$ be a finite subgroup of $\GL(d,\C)$.
We shall say that $f$ \emph{commutes} with $\Gamma$ if  $f \Gamma = \Gamma f$, i.e. if
there exists a group isomorphism $\rho : \Gamma \to \Gamma$ such that 
$f \circ \gamma = \rho( \gamma) \circ f$ for all $\gamma \in \Gamma$.
A stronger condition is that $f$ commutes with \emph{all} elements of $\Gamma$ which is the case exactly when $\rho = \id$.

We call a group \emph{diagonalizable} if it is conjugated to a subgroup of the group of diagonal matrices.
Observe that any finite abelian group is diagonalizable.

Finally we recall the definition of a Poincar\'e-Dulac normal form.
\begin{defi}
Let $f:(\C^d,0) \rightarrow (\C^d,0)$ an attracting invertible germ. Denote by $\lambda=(\lambda_1, \ldots, \lambda_d)$ the  eigenvalues of $df(0)$ counted with multiplicities.
A monomial $x^n=x_1^{n_1} \cdots x_d^{n_d}$ with $n_1 + \cdots + n_d \geq 2$ is called \emph{resonant} for the $k$-th coordinate if $\lambda^n=\lambda_k$.

The germ $f$ is in \emph{Poincar\'e-Dulac normal form} if $df(0)$ is in Jordan normal form and the map $f(x) - df(0)\cdot x$ admits only resonant monomials.
\end{defi}

\begin{thm}\label{thm:PDcommutingGammadiag}
Let $\Gamma$ be a finite diagonalizable subgroup of $\GL(d,\C)$.
Suppose $f:(\C^d,0) \rightarrow (\C^d,0)$ is an attracting automorphism that commutes with all
elements of $\Gamma$.

Then there exists coordinates $x=(x_1, \ldots, x_d)$ at $0$ such that any element $g \in \Gamma$ acts as a diagonal linear map, and $f$ is in a Poincar\'e-Dulac normal form that commutes with all elements of $\Gamma$.
\end{thm}

\begin{rmk}\label{rmk:invariantconjugation}
Suppose there exists an analytic local diffeomorphism $\Phi:(\C^d,0)\rightarrow (\C^d,0)$ conjugating $f$ to holomorphic germ $\wt{f}$
that also commutes with $\Gamma$.  Then $f$ and $\wt{f}$ are conjugated by a diffeomorphism $\Psi$ that also commutes with $\Gamma$.

To see this introduce $\rho$ and $\tilde{\rho}$ the group automorphisms such that $f \circ g = \rho(g) \circ f$, and $\wt{f} \circ g = \tilde{\rho}(g) \circ \wt{f}$ respectively.
Then the map 
$$
\Psi=\frac{1}{\abs{\Gamma}} \sum_{g \in \Gamma} \wt{\rho}(g^{-1}) \circ \Phi \circ \rho(g)~.
$$
satisfies  $\Psi \circ g = (\wt{\rho} \circ \rho^{-1}) (g) \circ \Psi$ as required.
\end{rmk}
%

\begin{proof}[Proof of Theorem \ref{thm:PDcommutingGammadiag}]

Since $\Gamma$ is diagonalizable and $f$ commutes with all elements of $\Gamma$,  
we may assume that any $\gamma \in \Gamma$ is diagonal and $df(0)$ is in (lower triangular) Jordan normal form.
Write $f(x)=(f_1(x), \ldots, f_d(x))$,  $f_k(x)=\sum_n f_{k,n} x^n$ with  $f_{n,k} \in \C$ and $n\in\N^d$. 
If
$$
\gamma =\on{Diag} ( \zeta^{q_1}, \ldots, \zeta^{q_d}),
$$
with $\zeta$ a primitive $p$-th root of unity, and $q = (q_1, \ldots, q_d) \in \N^d$, 
then $f \circ \gamma = \gamma \circ f$ if and only if 
\begin{equation}\label{eqn:diaginvres}
q \cdot n \equiv q_k  \mod p~,
\end{equation}
for all $k,n$ such that $f_{k,n} \neq 0$.

Following the standard scheme for conjugating $f$ to a Poincar\'e-Dulac normal form, we show by induction that for any integer $N\ge1$ there exists
a Poincar\'e-Dulac normal form $\tilde{f}_N$ and a local biholomorphism $\Phi_N$ that both commute with all elements of $\Gamma$ such that $ \Phi_N \circ f - \tilde{f}_N \circ \Phi_N =  \cO(x^{N+1})$. This claim implies the theorem since it is known that for $N$ large enough $ \tilde{f}_N^{-n} \circ f^{n} $ converges to a conjugacy $\Phi$ between the two, see e.g.~\cite[p.84--85]{rosay-rudin:holomorphicmaps}, and this conjugacy commutes with all elements of $\Gamma$ by construction.

For $N=1$, pick $\tilde{f}_1 = df(0)\cdot x$ and $\Phi_1 = \id$. Suppose the claim has been proved for $N-1$. 
Denote by $\mathcal{H}_N$ the space of polynomial maps $\C^d \to \C^d$ whose coordinates are homogeneous polynomials of degree $N$, and $\mathcal{H}_N^\Gamma\subset \mathcal{H}_N$ for those commuting with all elements of $\Gamma$. These are finite dimensional vector spaces. Write $X_N$ for the subspace of $\mathcal{H}_N$ of those maps having only resonant monomials, and
$X_N^\Gamma := X_N\cap \mathcal{H}_N^\Gamma$. 

By~\cite[Lemma 2]{rosay-rudin:holomorphicmaps} any element $G\in \mathcal{H}_N$ can be written as
\begin{equation}\label{bougre}
G =  S + H \circ df(0) - df(0) \circ H
\end{equation}
for some $S \in X_N$ and $H\in \mathcal{H}_N$. If $G\in \mathcal{H}_N^\Gamma$, then we may replace $S$ and $H$ by
$\frac1{|\Gamma|} \sum \gamma^{-1} \circ S \circ \gamma$ and $\frac1{|\Gamma|} \sum \gamma^{-1} \circ H \circ \gamma$, so that 
we can solve~\eqref{bougre} with $S \in X_N^\Gamma$ and $H\in \mathcal{H}_N^\Gamma$.

Decompose $f = f_{< N} + G + \cO (|x|^{N+1})$ where $G$ is the homogeneous part of $f$ of degree $N$, 
and $f_{<N}$ contains only monomials of degree $<N$. Observe that by assumption $f_{<N}$ is in Poincar\'e-Dulac normal form.
The claim immediately follows with $\tilde{f}_N = f_{<N} + S$ and $\Phi_N := \id - H$.
\end{proof}

\bibliographystyle{alpha}
\bibliography{biblio}

\def\cprime{$'$} \def\cprime{$'$} \def\cprime{$'$} \def\cprime{$'$}
\begin{thebibliography}{BHPVdV04}

\bibitem[BdFF12]{boucksom-defernex-favre:volumeisolatedsing}
Sebastien Boucksom, Tommaso de~Fernex, and Charles Favre.
\newblock The volume of an isolated singularity.
\newblock {\em Duke Math. J.}, 161(8):1455--1520, 2012.

\bibitem[BHPVdV04]{barth-hulek-peters-vanderven:compactcomplexsurfaces}
Wolf~P. Barth, Klaus Hulek, Chris A.~M. Peters, and Antonius Van~de Ven.
\newblock {\em Compact complex surfaces}, volume~4 of {\em Ergebnisse der
  Mathematik und ihrer Grenzgebiete. 3. Folge. A Series of Modern Surveys in
  Mathematics [Results in Mathematics and Related Areas. 3rd Series. A Series
  of Modern Surveys in Mathematics]}.
\newblock Springer-Verlag, Berlin, second edition, 2004.

\bibitem[BN51]{bundgaard-nielsen:normalsubgroupsfiniteindex}
Svend Bundgaard and Jakob Nielsen.
\newblock On normal subgroups with finite index in {$F$}-groups.
\newblock {\em Mat. Tidsskr. B.}, 1951:56--58, 1951.

\bibitem[Can99]{cantat:dynamiqueautosurfproj}
Serge Cantat.
\newblock Dynamique des automorphismes des surfaces projectives complexes.
\newblock {\em C. R. Acad. Sci. Paris S\'er. I Math.}, 328(10):901--906, 1999.

\bibitem[CMS09]{camacho-movasati-scardua:quasihomosteinsurfsing}
C.~Camacho, H.~Movasati, and B.~Sc{\'a}rdua.
\newblock The moduli of quasi-homogeneous {S}tein surface singularities.
\newblock {\em J. Geom. Anal.}, 19(2):244--260, 2009.

\bibitem[Dem93]{demailly:mongeamperelelongnumbersintertheory}
Jean-Pierre Demailly.
\newblock Monge-{A}mp\`ere operators, {L}elong numbers and intersection theory.
\newblock In {\em Complex analysis and geometry}, Univ. Ser. Math., pages
  115--193. Plenum, New York, 1993.

\bibitem[Fox52]{fox:fenchelconjFgroups}
Ralph~H. Fox.
\newblock On {F}enchel's conjecture about {$F$}-groups.
\newblock {\em Mat. Tidsskr. B.}, 1952:61--65, 1952.

\bibitem[FS92]{furuta-steer:seifertfibredhomology}
Mikio Furuta and Brian Steer.
\newblock Seifert fibred homology {$3$}-spheres and the {Y}ang-{M}ills
  equations on {R}iemann surfaces with marked points.
\newblock {\em Adv. Math.}, 96(1):38--102, 1992.

\bibitem[Giz80]{gizatullin:rationalGsurfaces}
M.~H. Gizatullin.
\newblock Rational {$G$}-surfaces.
\newblock {\em Izv. Akad. Nauk SSSR Ser. Mat.}, 44(1):110--144, 239, 1980.

\bibitem[Gra62]{grauert:ubermodifikationen}
Hans Grauert.
\newblock \"{U}ber {M}odifikationen und exzeptionelle analytische {M}engen.
\newblock {\em Math. Ann.}, 146:331--368, 1962.

\bibitem[Hol60]{holmann:quotkomplexautogrup}
Harald Holmann.
\newblock Quotientenr\"aume komplexer {M}annigfaltigkeiten nach komplexen
  {L}ieschen {A}utomorphismengruppen.
\newblock {\em Math. Ann.}, 139:383--402 (1960), 1960.

\bibitem[HP11]{horing-peternell:nonalgcpctkahler3fldsendo}
Andreas H{\"o}ring and Thomas Peternell.
\newblock Non-algebraic compact {K}\"ahler threefolds admitting endomorphisms.
\newblock {\em Sci. China Math.}, 54(8):1635--1664, 2011.

\bibitem[Kat79]{kato:cpctcplxsurfwithGSPH}
Masahide Kato.
\newblock Compact complex surfaces containing global strongly pseudoconvex
  hypersurfaces.
\newblock {\em T\^ohoku Math. J. (2)}, 31(4):537--547, 1979.

\bibitem[Lau71]{laufer:normal2dimsing}
Henry~B. Laufer.
\newblock {\em Normal two-dimensional singularities}.
\newblock Princeton University Press, Princeton, N.J., 1971.
\newblock Annals of Mathematics Studies, No. 71.

\bibitem[Mil85]{milnor:conceptattractor}
John Milnor.
\newblock On the concept of attractor.
\newblock {\em Comm. Math. Phys.}, 99(2):177--195, 1985.

\bibitem[Mil06]{milnor:dyn1cplxvar}
John Milnor.
\newblock {\em Dynamics in one complex variable}, volume 160 of {\em Annals of
  Mathematics Studies}.
\newblock Princeton University Press, Princeton, NJ, third edition, 2006.

\bibitem[M{\"u}l87]{muller:liegroupsanalyticalgebras}
Gerd M{\"u}ller.
\newblock Actions of complex {L}ie groups on analytic {${\bf C}$}-algebras.
\newblock {\em Monatsh. Math.}, 103(3):221--231, 1987.

\bibitem[M{\"u}l99]{muller:symmetriessurfsing}
Gerd M{\"u}ller.
\newblock Symmetries of surface singularities.
\newblock {\em J. London Math. Soc. (2)}, 59(2):491--506, 1999.

\bibitem[M{\"u}l00]{muller:resolweighethomosurfsing}
Gerd M{\"u}ller.
\newblock Resolution of weighted homogeneous surface singularities.
\newblock In {\em Resolution of singularities ({O}bergurgl, 1997)}, volume 181
  of {\em Progr. Math.}, pages 507--517. Birkh\"auser, Basel, 2000.

\bibitem[Nak]{nakayama:complexnormalprojsurfnonisosurjendo}
Noboru Nakayama.
\newblock On complex normal projective surfaces admitting non-isomorphic
  surjective endomorphisms.
\newblock Preprint.

\bibitem[NZ09]{nakayama-zhang:buildingblocksetaleendocplxprojmfld}
Noboru Nakayama and De-Qi Zhang.
\newblock Building blocks of \'etale endomorphisms of complex projective
  manifolds.
\newblock {\em Proc. Lond. Math. Soc. (3)}, 99(3):725--756, 2009.

\bibitem[NZ10]{nakayama-zhang:polarizedendocplxnormalvar}
Noboru Nakayama and De-Qi Zhang.
\newblock Polarized endomorphisms of complex normal varieties.
\newblock {\em Math. Ann.}, 346(4):991--1018, 2010.

\bibitem[Oda88]{oda:convexbodies}
Tadao Oda.
\newblock {\em Convex bodies and algebraic geometry}, volume~15 of {\em
  Ergebnisse der Mathematik und ihrer Grenzgebiete (3) [Results in Mathematics
  and Related Areas (3)]}.
\newblock Springer-Verlag, Berlin, 1988.
\newblock An introduction to the theory of toric varieties, Translated from the
  Japanese.

\bibitem[OW71]{orlik-wagreich:isolatedsingagsurfCaction}
Peter Orlik and Philip Wagreich.
\newblock Isolated singularities of algebraic surfaces with {C{$^{\ast}$}}\
  action.
\newblock {\em Ann. of Math. (2)}, 93:205--228, 1971.

\bibitem[Pin84]{pinkham:autocuspsinouehirzebruchsurf}
H.~C. Pinkham.
\newblock Automorphisms of cusps and {I}noue-{H}irzebruch surfaces.
\newblock {\em Compositio Math.}, 52(3):299--313, 1984.

\bibitem[RR88]{rosay-rudin:holomorphicmaps}
Jean-Pierre Rosay and Walter Rudin.
\newblock Holomorphic maps from {${\bf C}^n$} to {${\bf C}^n$}.
\newblock {\em Trans. Amer. Math. Soc.}, 310(1):47--86, 1988.

\bibitem[RS05]{rashkovskii-sigurdsson:greenfcntwithsing}
Alexander Rashkovskii and Ragnar Sigurdsson.
\newblock Green functions with singularities along complex spaces.
\newblock {\em Internat. J. Math.}, 16(4):333--355, 2005.

\bibitem[RT11]{ross-thomas:weightedprojembeddingorbifoldscscKahler}
Julius Ross and Richard Thomas.
\newblock Weighted projective embeddings, stability of orbifolds, and constant
  scalar curvature {K}\"aher metrics.
\newblock {\em J. Differential Geom.}, 88(1):109--159, 2011.

\bibitem[San87]{sankaran:higherdiminouehirzebruchsurfaces}
G.~K. Sankaran.
\newblock Higher-dimensional analogues of {I}noue-{H}irzebruch surfaces.
\newblock {\em Math. Ann.}, 276(3):515--528, 1987.

\bibitem[Sat57]{satake:gaussbonnetVmflds}
Ichir{\^o} Satake.
\newblock The {G}auss-{B}onnet theorem for {$V$}-manifolds.
\newblock {\em J. Math. Soc. Japan}, 9:464--492, 1957.

\bibitem[SB93]{sanchezbringas:normalformsinvariantvectorfields}
Federico S{\'a}nchez-Bringas.
\newblock Normal forms of invariant vector fields under a finite group action.
\newblock {\em Publ. Mat.}, 37(1):75--82, 1993.

\bibitem[Sco83]{scott:geometries3mflds}
Peter Scott.
\newblock The geometries of {$3$}-manifolds.
\newblock {\em Bull. London Math. Soc.}, 15(5):401--487, 1983.

\bibitem[Sel60]{selberg:discgroupssymspaces}
Atle Selberg.
\newblock On discontinuous groups in higher-dimensional symmetric spaces.
\newblock In {\em Contributions to function theory (internat. {C}olloq.
  {F}unction {T}heory, {B}ombay, 1960)}, pages 147--164. Tata Institute of
  Fundamental Research, Bombay, 1960.

\bibitem[SW81]{scheja-wiebe:chevalleyzerlegungderiv}
G{\"u}nter Scheja and Hartmut Wiebe.
\newblock Zur {C}hevalley-{Z}erlegung von {D}erivationen.
\newblock {\em Manuscripta Math.}, 33(2):159--176, 1980/81.

\bibitem[Tsu83]{tsuchihashi:higherdimcuspsing}
Hiroyasu Tsuchihashi.
\newblock Higher-dimensional analogues of periodic continued fractions and cusp
  singularities.
\newblock {\em Tohoku Math. J. (2)}, 35(4):607--639, 1983.

\bibitem[Wag83]{wagreich:structurequasihomogensing}
Philip Wagreich.
\newblock The structure of quasihomogeneous singularities.
\newblock In {\em Singularities, {P}art 2 ({A}rcata, {C}alif., 1981)},
  volume~40 of {\em Proc. Sympos. Pure Math.}, pages 593--611. Amer. Math.
  Soc., Providence, RI, 1983.

\bibitem[Wah90]{wahl:charnumlinksurfsing}
Jonathan Wahl.
\newblock A characteristic number for links of surface singularities.
\newblock {\em J. Amer. Math. Soc.}, 3(3):625--637, 1990.

\bibitem[Zha13]{zhang:algvarautogroupsmaxrank}
De-Qi Zhang.
\newblock Algebraic varieties with automorphism groups of maximal rank.
\newblock {\em Math. Ann.}, 355(1):131--146, 2013.

\end{thebibliography}

\end{document}